\documentclass{article}

\usepackage{makeidx}
\usepackage{graphicx}
\usepackage{multicol}
\usepackage[bottom]{footmisc}
\usepackage{amsmath,amssymb,amsfonts, amsthm}
\usepackage[affil-it]{authblk}

\makeindex

\newcommand{\Z}{\mathbb Z}

\newcommand{\su}{\subseteq}

\newtheorem{thm}{Theorem}[section]
\newtheorem{lem}[thm]{Lemma}

\newtheorem{obs}[thm]{Observation}

\theoremstyle{definition}
\newtheorem{defn}[thm]{Definition}

\theoremstyle{remark}
\newtheorem{rem}[thm]{Remark}
\newtheorem{case}{Case}[thm]

\author{Santanu Mondal, Krishnendu Paul, Shameek Paul%
\thanks{E-mail addresses: \texttt{santanu.mondal.math18@gm.rkmvu.ac.in, krishnendu.p.math18@gm.rkmvu.ac.in, shameek.paul@rkmvu.ac.in}}}
\affil{\small Ramakrishna Mission Vivekananda Educational and Research Institute, P.O. Belur Math, Dist. Howrah, 711202, India}

\date{}
\title {Extremal sequences related to the Jacobi symbol}

\begin{document}

\baselineskip=14.5pt

\maketitle

\begin{abstract}
For a weight-set $A\su \Z_n$, the $A$-weighted zero-sum constant $C_A(n)$ is defined to be the smallest natural number $k$, such that any sequence of $k$ elements in $\Z_n$ has an $A$-weighted zero-sum subsequence of consecutive terms. A sequence of length $C_A(n)-1$ in $\Z_n$ which does not have any $A$-weighted zero-sum subsequence of consecutive terms will be called a $C$-extremal sequence for $A$. 

Let $\big(\frac{x}{n}\big)$ denote the Jacobi symbol of $x\in\Z_n$. We characterize the $C$-extremal sequences for the weight-set $S(n)=\big\{\,x\in U(n):\big(\frac{x}{n}\big)=1\,\big\}$ and for the weight-set $L(n;p)=\big\{\,x\in U(n):\big(\frac{x}{n}\big)=\big(\frac{x}{p}\big)\,\big\}$ where $p$ is a prime divisor of $n$. We can define $D$-extremal sequences for these weight-sets in a way analogous to the definition of $C$-extremal sequences. We also  characterize these sequences. 
\end{abstract}

\bigskip


\bigskip


\vspace{.3cm}

\section{Introduction}\label{0}

This paper is a complementary paper to \cite{SKS3} and we have used a few results from \cite{SKS3}. We recall some notations which have been used in that paper in this section. 

\begin{defn} 
For a subset $A \su \Z_n$, the $A$-weighted Davenport constant $D_A(n)$, is defined to be the least positive integer $k$ such that any sequence in $\Z_n$ of length $k$ has an $A$-weighted zero-sum subsequence. 
\end{defn}

This constant was introduced in \cite{AR}.

\begin{defn} 
For a subset $A \su \Z_n$, the $A$-weighted zero-sum constant $C_A(n)$, is defined to be the least positive integer $k$ such that any sequence in $\Z_n$ of length $k$ has an $A$-weighted zero-sum subsequence of consecutive terms. 
\end{defn}

This constant was introduced in \cite{SKS1}. Observe that $D_A(n)=1$ if and only if $0\in A$. The same is true for the constant $C_A(n)$.

\begin{defn}
Let $0\notin A$. Then there exists a sequence $S$ in $\Z_n$ of length $D_A(n)-1$ which has no $A$-weighted zero-sum subsequence. We will call such a sequence a $D$-extremal sequence for $A$.  
\end{defn}

This definition was given in \cite{AMP} where such sequences were characterized for some weight sets. The following definition was given in \cite{SKS2}.

\begin{defn}
Let $0\notin A$. Then there exists a sequence $S$ in $\Z_n$ of length $C_A(n)-1$ which has no $A$-weighted zero-sum subsequence of consecutive terms. We will call such a sequence a $C$-extremal sequence for $A$. 
\end{defn}

Let $U(n)$ denote the multiplicative group of units in the ring $\Z_n$. Let $U(n)^2$ denote the subgroup $\{\,x^2:x\in U(n)\,\}$. For an odd prime $p$, let $Q_p$ denote the set $U(p)^2$. If $p$ is a prime divisor of $n$, we use the notation $v_p(n)=r$ to mean that $p^r\mid n$ and $p^{r+1}\nmid n$. If $n$ is squarefree, then we define $\Omega(n)$ to be the number of prime divisors of $n$. 

The following are some of the results which have been proved in this paper. Let $\big(\frac{x}{n}\big)$ denote the Jacobi symbol of $x\in\Z_n$, $S(n)=\big\{\,x\in U(n):\big(\frac{x}{n}\big)=1\,\big\}$ and $L(n;p)=\big\{\,x\in U(n):\big(\frac{x}{n}\big)=\big(\frac{x}{p}\big)\,\big\}$ where $p$ is a prime divisor of $n$. Assume that $n$ is a squarefree number such that every prime divisor of $n$ is at least 7 and $p$ is a prime divisor of $n$. For a sequence $S$ in $\Z_n$ we have proved the following.  

\begin{itemize}
\item 
Suppose $\Omega(n)\neq 1$. Then $S$ is a $C$-extremal sequence for $S(n)$ if and only if $S$ is a $C$-extremal sequence for $U(n)$. 

\item 
Suppose $\Omega(n)\neq 1,2$. Then $S$ is a $D$-extremal sequence for $S(n)$ if and only if $S$ is a $D$-extremal sequence for $U(n)$.

\item 
Suppose $\Omega(n)\neq 2$. Then $S$ is a $C$-extremal sequence for $L(n;p)$ if and only if $S$ is a $C$-extremal sequence for $U(n)$.

\item 
Suppose $\Omega(n)\neq 2,3$. Then $S$ is a $D$-extremal sequence for $L(n;p)$ if and only if $S$ is a $D$-extremal sequence for $U(n)$.
\end{itemize}

\begin{rem}
When $n$ is odd, the $D$-extremal sequences for $U(n)$ have been characterized in Theorem 6 of \cite{AMP} and the $C$-extremal sequences for $U(n)$ have been characterized in Theorems 5 and 6 of \cite{SKS2}. 
\end{rem}

The next two results are given in \cite{SKS1} and \cite{SKS3}. 

\begin{thm}\label{un}
For $n$ odd, we have $D_{U(n)}(n)=\Omega(n)+1$ and $C_{U(n)}(n)=2^{\Omega(n)}$. 
\end{thm}

\begin{thm}\label{qp}
For an odd prime $p$, we have $D_{Q_p}(p)=C_{Q_p}(p)=3$. 
\end{thm}

Let $m$ be a divisor of $n$. We will refer to the homomorphism $f_{n\mid m}:\Z_n\to \Z_m$ given by $a+n\Z\mapsto a+m\Z$ as the natural map. Clearly this map is onto. As the image of $U(n)$ under $f_{n|m}$ is in $U(m)$, we get a map $U(n)\to U(m)$ which we will also refer to as the natural map and denote by $f_{n|m}$.
 
Let $S=(x_1,\ldots,x_l)$ be a sequence in $\Z_n$ and let $p$ be a prime divisor of $n$ such that $v_p(n)=r$. Let us denote the image of $x\in \Z_n$ under $f_{n|p^r}$ by $x^{(p)}$ and let $S^{(p)}$ denote the sequence $(x_1^{(p)},\ldots,x_l^{(p)})$ in $\Z_{p^r}$.

\section{Some results about the weight-set $S(n)$}\label{s}

From this point onwards, we will always assume that $n$ is odd. For $x\in\Z_n$ the Jacobi symbol $\big(\frac{x}{n}\big)$ has been defined in \cite{SKS3}. 

\begin{defn}\label{jacobi}
$S(n)=\big\{\,x\in U(n):\big(\frac{x}{n}\big)=1\,\big\}$ where $\big(\frac{x}{n}\big)$ is the Jacobi symbol.  
\end{defn}

\begin{rem}
Considering $S(n)$ as a weight set was suggested in Section 3 of \cite{ADU}. The constants $D_{S(n)}(n)$ and $C_{S(n)}(n)$ were computed in \cite{SKS3}. From Proposition 1 of \cite{SKS3} we see that $S(n)=U(n)$ if $n$ is a square and  $S(n)$ is a subgroup of index two in $U(n)$, otherwise. When $p$ is an odd prime, it follows that $S(p)=Q_p$. 
\end{rem}

We now state some results from \cite{SKS3} which will be used in the proofs in the next section. The next four results are Lemmas 3, 4, 8 and 9 of \cite{SKS3}.

\begin{lem}\label{u2s}
Let $d$ be a proper divisor of $n$ which is not a square. Let $n'=n/d$ and suppose that $d$ is coprime with $n'$. Then $U(n')\su f_{n|n'}\big(S(n)\big)$.
\end{lem}

\begin{lem}\label{lifts'}
Let $S$ be a sequence in $\Z_n$ and let $d$ be a proper divisor of $n$ which divides every element of $S$. Let $m=n/d$ and suppose that $d$ is coprime with $m$. Let $S'$ be the image of $S$ under $f_{n|m}$. Let $A\su\Z_n$ and $B\su f_{n|m}(A)$. Suppose $S'$ is a $B$-weighted zero-sum sequence. Then $S$ is an $A$-weighted zero-sum sequence.  
\end{lem}

\begin{lem}\label{gs}
Let $n$ be an odd, squarefree number and $S$ be a sequence in $\Z_n$ such that for every prime divisor $p$ of $n$ at least two terms of $S$ are coprime to $p$. Suppose at most one term of $S$ is a unit. Then $S$ is an $S(n)$-weighted zero-sum sequence. 
\end{lem}

\begin{lem}\label{gs'}
Let $n$ be a squarefree number whose every prime divisor is at least 7. Let $S$ be a sequence in $\Z_n$ such that for every prime divisor $p$ of $n$ at least two terms of $S$ are coprime to $p$. Suppose there is a prime divisor $p$ of $n$ such that at least three terms of $S$ are coprime to $p$. Then $S$ is an $S(n)$-weighted zero-sum sequence.
\end{lem}

The next result follows from Theorems 3 and 4 of \cite{SKS3}.

\begin{thm}\label{sn}
When $n$ is a prime, we have $D_{S(n)}(n)=C_{S(n)}(n)=3$. When $n$ is not a prime and is a squarefree number whose every prime divisor is at least 7, we have $D_{S(n)}(n)=\Omega(n)+1$ and $C_{S(n)}(n)=2^{\Omega(n)}$. 
\end{thm}

\section{$D$-extremal sequences for $S(n)$}

\begin{thm}\label{dexts}
Let $n$ be squarefree and every prime divisor of $n$ be at least 7. Suppose $S$ is a $D$-extremal sequence for $S(n)$. If $\Omega(n)\geq 3$, then $S$ is a $D$-extremal sequence for $U(n)$. If $\Omega(n)=2$, then $S$ is either a $D$-extremal sequence for $U(n)$ or $S$ is a permutation of a sequence $(x_1,x_2)$ where $x_1\in S(n)$ and $-x_2\in U(n)\setminus S(n)$.
\end{thm}

\begin{proof}
Let $S=(x_1,\ldots,x_k)$ be a $D$-extremal sequence for $S(n)$. Then all terms of $S$ are non-zero. As $n$ is squarefree, $\Omega(n)\geq 2$ and every prime divisor of $n$ is at least 7, by Theorem \ref{sn} we have $D_{S(n)}(n)=\Omega(n)+1$. Hence $k$ must be $\Omega(n)$. 

\begin{case}
There is a prime divisor $p$ of $n$ such that at most one term of $S$ is coprime to $p$.
\end{case}

Suppose all terms of $S$ are divisible by $p$. Let $n'=n/p$ and let $S'$ be the image of $S$ under $f_{n|n'}$. By Theorem $\ref{un}$ we have $D_{U(n')}(n')=\Omega(n')+1$. As $S'$ has length $\Omega(n)=\Omega(n')+1$, we see that $S'$ has a $U(n')$-weighted zero-sum subsequence. As $n$ is squarefree, it follows that $p$ is coprime to $n'$ and so by Lemmas \ref{u2s} and \ref{lifts'}, we get the contradiction that $S$ has an $S(n)$-weighted zero-sum subsequence.

So there is exactly one term of $S$ which is not divisible by $p$. Let us assume that that term is $x_1$. Let $T=(x_2,\ldots,x_k)$ and let $T'$ be the image of $T$ under $f_{n|n'}$. Suppose $T'$ has a $U(n')$-weighted zero-sum subsequence. By Lemmas \ref{u2s} and \ref{lifts'}, we get the contradiction that $S$ has an $S(n)$-weighted zero-sum subsequence. Thus, $T'$ is a sequence of length $\Omega(n')$ in $\Z_{n'}$ which does not have any $U(n')$-weighted zero-sum subsequence. As $D_{U(n')}(n')=\Omega(n')+1$, it follows that $T'$ is a $D$-extremal sequence for $U(n')$. So from Theorem 5 of \cite{AMP} we see that $S$ is a $D$-extremal sequence for $U(n)$.

\begin{case}
For every prime divisor $p$ of $n$, exactly two terms of $S$ are coprime to $p$.
\end{case}

If $\Omega(n)=2$, then $S=(x_1,x_2)$. For every prime divisor $p$ of $n$, as both $x_1$ and $x_2$ are coprime to $p$, it follows that $x_1,x_2\in U(n)$. Suppose either $x_1,-x_2\in S(n)$ or $x_1,-x_2\notin S(n)$. As $n$ is squarefree, from Proposition 1 of \cite{SKS3} we get that $S(n)$ has index two in $U(n)$. Hence, $a\in S(n)$ where  $a=-x_2x_1^{-1}$. As $1\in S(n)$ and $ax_1+x_2=0$, it follows that $S$ is an $S(n)$-weighted zero-sum sequence. Thus, the sequence $S$ is a permutation of a sequence $(x_1,x_2)$ where $x_1\in S(n)$ and $-x_2\in U(n)\setminus S(n)$.
  
Let us now assume that $\Omega(n)\geq 3$. Suppose $S$ has at most one unit. By Lemma \ref{gs} we get the contradiction that $S$ is an $S(n)$-weighted zero-sum sequence. So we can assume that $S$ has at least two units. By the assumption in this subcase, we see that $S$ will have exactly two units and the other terms of $S$ will be zero. As $S$ has length $k\geq 3$, we get the contradiction that some term of $S$ is zero. 

\begin{case}
For every prime divisor $p$ of $n$ at least two terms of $S$ are coprime to $p$, and there is a prime divisor $p$ of $n$ such that at least three terms of $S$ are coprime to $p$.
\end{case}

In this case, by Lemma \ref{gs'} we get the contradiction that $S$ is an $S(n)$-weighted zero-sum sequence. 
\end{proof}

\begin{thm}\label{dexts'}
Let $n$ be a squarefree number which is not a prime and suppose every prime divisor of $n$ is at least 7. When $\Omega(n)=2$, a sequence is a $D$-extremal sequence for $S(n)$ if and only if either it is a $D$-extremal sequence for $U(n)$ or a permutation of a sequence $(x_1,x_2)$ where $x_1\in S(n)$ and $-x_2\in U(n)\setminus S(n)$. When $\Omega(n)\geq 3$, a sequence is $D$-extremal sequence for $S(n)$ if and only if  it is a $D$-extremal sequence for $U(n)$.
\end{thm}

\begin{proof}
Let $n$ be as in the statement of the theorem. We have already proved one implication in Theorem \ref{dexts}. From Theorems \ref{un} and \ref{sn} we have $D_{U(n)}(n)=D_{S(n)}(n)$. Hence, as $S(n)\su U(n)$ it follows that if $S$ is a $D$-extremal sequence for $U(n)$, then $S$ is a $D$-extremal sequence for $S(n)$.  

Let $\Omega(n)=2$ and $S=(x_1,x_2)$ where $x_1\in S(n)$ and $-x_2\in U(n)\setminus S(n)$. Suppose $T$ is an $S(n)$-weighted zero-sum subsequence of $S$. Then $T$ must be $S$ itself. So there exist $a,b\in S(n)$ such that $ax_1+bx_2=0$ and hence there exists $c\in S(n)$ such that $-x_2=cx_1$. As both $x_1$ and $c$ are in $S(n)$, we get the contradiction that $-x_2\in S(n)$. Thus $S$ does not have any $S(n)$-weighted zero-sum subsequence. From Theorem \ref{sn}, we have $D_{S(n)}(n)=3$ and so $S$ is a $D$-extremal sequence for $S(n)$. 
\end{proof}

\begin{rem}
When $n$ is a prime $p$ we have $S(n)=Q_p$. From Corollary 2 of \cite{SKS2}, we can see that the $D$-extremal sequences for $Q_p$ are precisely those which are a permutation of a sequence of the form $(x_1,x_2)$ where $x_1\in Q_p$ and $-x_2\in U(p)\setminus Q_p$.
\end{rem}


\section{$C$-extremal sequences for $S(n)$}

\begin{thm}\label{cexts}
Let $n$ be squarefree number which is not a prime and assume that every prime divisor of $n$ is at least 7. Then a sequence  in $\Z_n$ is a $C$-extremal sequence for $S(n)$ if and only if it is a $C$-extremal sequence for $U(n)$. 
\end{thm}

\begin{proof}
As $n$ is a squarefree number which is not a prime and every prime divisor of $n$ is at least 7, by Theorems \ref{un} and \ref{sn} we get $C_{U(n)}(n)=C_{S(n)}(n)$. Hence, as $S(n)\su U(n)$, it follows that any $C$-extremal sequence for $U(n)$ is a $C$-extremal sequence for $S(n)$. Suppose a sequence $S=(x_1,\ldots,x_l)$ in $\Z_n$ is a $C$-extremal sequence for $S(n)$. By Theorem \ref{sn} we have $C_{S(n)}(n)=2^{\Omega(n)}$ and so   $l$ must be $2^{\Omega(n)}-1$. Also all the terms of $S$ must be non-zero. 

\begin{case}
There is a prime divisor $p$ of $n$ such that at most one term of $S$ is not divisible by $p$.
\end{case}

Suppose all the terms of $S$ are divisible by $p$. Let $n'=n/p$ and $S'$ be the image of $S$ under $f_{n|n'}$. By Theorem \ref{un} we have $C_{U(n')}(n')=2^{\Omega(n')}$. As $S'$ has length $l=2^{\Omega(n)}-1$ and as $\Omega(n)=\Omega(n')+1$, we get that $l>2^{\Omega(n')}$. Hence, $S'$ has a $U(n')$-weighted zero-sum subsequence of consecutive terms. Thus by Lemmas \ref{u2s} and \ref{lifts'}, we get the contradiction that $S$ has an $S(n)$-weighted zero-sum subsequence of consecutive terms. 

So exactly one term $x^*$ of $S$ is coprime to $p$. Suppose $x^*\neq x_{k+1}$ where $k+1=(l+1)/2$. Then there is a subsequence $T$ of consecutive terms of $S$ of length at least $k+1$ such that $p$ divides every term of $T$. Also $k+1=2^{\Omega(n')}$ as $l+1=2^{\Omega(n)}$. So by a similar argument as in the previous paragraph, we get the contradiction that $T$ (and hence $S$) has an $S(n)$-weighted zero-sum subsequence of consecutive terms. Thus the term $x^*$ must be $x_{k+1}$. 

Let $S_1'$ and $S_2'$ be the images of the sequences $S_1=(x_1,\ldots,x_k)$ and $S_2=(x_{k+2},\ldots,x_l)$ under $f_{n|n'}$. Suppose $S_1'$ has an $S(n')$-weighted zero-sum subsequence of consecutive terms. By Lemma \ref{u2s}, we have $U(n')\su f_{n|n'}\big(S(n)\big)$. As $p$ divides every term of $S_1$, by Lemma \ref{lifts'} we see that $S_1$ has an $S(n)$-weighted zero-sum subsequence of consecutive terms. Now as $S_1$ is itself a subsequence of consecutive terms of $S$, we get the contradiction that $S$ has an $S(n)$-weighted zero-sum subsequence of consecutive terms. 

By Theorem \ref{un}, we have $C_{U(n')}(n')=2^{\Omega(n')}$. As $S_1'$ does not have any $U(n')$-weighted zero-sum subsequence of consecutive terms and $S_1'$ has length $k=2^{\Omega(n')}-1$, it follows that $S_1'$ is a $C$-extremal sequence for $U(n')$ in $\Z_{n'}$. A similar argument shows that $S_2'$ is also a $C$-extremal sequence for $U(n')$ in $\Z_{n'}$. So from Theorem 5 of \cite{SKS2}, we see that $S$ is a $C$-extremal sequence for $U(n)$.

\begin{case}
Given any prime divisor of $n$, exactly two terms of $S$ are not divisible by it.
\end{case}

If $S$ has at most one unit, by Lemma \ref{gs} we get the contradiction that $S$ is an $S(n)$-weighted zero-sum sequence. So we can assume that $S$ has at least two units. By the assumption in this subcase, we see that $S$ will have exactly two units and the other terms of $S$ will be zero. As $\Omega(n)\geq 2$ and as $l=2^{\Omega(n)}-1$, we see that $S$ has at least three terms. Thus, we get the contradiction that $S$  has a term which is zero. 

\begin{case}
Given any prime divisor of $n$ at least two terms of $S$ are not divisible by it, and there is a prime divisor of $n$ such that at least three terms of $S$ are not divisible by it.
\end{case}

As $n$ is squarefree and every prime divisor of $n$ is at least 7, by Lemma \ref{gs'} we get the contradiction that $S$ is an $S(n)$-weighted zero-sum sequence.
\end{proof}

\begin{rem}
When $n$ is a prime $p$, we have that $S(n)=Q_p$. The $C$-extremal sequences for $Q_p$ have been characterized in Corollary 2 of \cite{SKS2}. They are the sequences which are a permutation of a sequence of the form $(x_1,x_2)$ where $x_1\in Q_p$ and $-x_2\in U(p)\setminus Q_p$.  
\end{rem}

\section{Some results about the weight-set $L(n;p)$}

In an attempt in \cite{SKS3} to determine the constant $C_{S(n)}(n)$ when $n$ is not squarefree for a prime divisor $p$ of $n$ we had considered the subset $L(n;p)$ of $\Z_n$. 

\begin{defn}\label{lnp}
For a prime divisor $p$ of $n$, let 
$$L(n;p)=\Big\{\,a\in U(n)\,\Big{|}\,\Big(\dfrac{a}{n}\Big) =\Big(\dfrac{a}{p}\Big)\Big\}$$
\end{defn}

\begin{rem}
From Proposition 2 of \cite{SKS3}, we see that $L(n;p)=U(n)$ if $n$ has a unique prime divisor $p$ such that $v_p(n)$ is odd, and $L(n;p)$ is a subgroup of $U(n)$ having index two, otherwise. If $n'=n/p$, then $f_{n|n'}\big(L(n;p)\big)\su S(n')$ as for any $a\in U(n)$ we have
$$\Big(\dfrac{a}{n}\Big) =\Big(\dfrac{a}{n'}\Big)\Big(\dfrac{a}{p}\Big).$$
\end{rem}

The next five results, which are Lemmas 10, 11, 12 and 13 and Observation 3 of [8], will be used in the proofs in the next section.

\begin{lem}\label{s2l}
Let $p'$ and $p$ be prime divisors of $n$. Let $n'=n/p$ and suppose that $n'$ is coprime with $p$. Then $S(n')\su f_{n|n'}\big(L(n;p')\big)$. 
\end{lem}

\begin{lem}\label{u2l}
Let $p'$ be a prime divisor of $n$ which is coprime to $n'=n/p'$. Then $U(p')\su f_{n|p'}\big(L(n;p')\big)$.  
\end{lem}

\begin{lem}\label{gl'}
Let $n$ be squarefree and let $p'$ be a prime divisor of $n$. Let $n'=n/p'$ and let $\psi:U(n)\to U(n')\times U(p')$ be the isomorphism given by the Chinese remainder theorem. Then $S(n')\times U(p')\su \psi\big(L(n;p')\big)$.
\end{lem}

\begin{lem}\label{gl}
Let $n$ be squarefree and $p'$ be a prime divisor of $n$. Let $S$ be a sequence in $\Z_n$ such that for every prime divisor $p$ of $n$, at least two terms of $S$ are coprime to $p$. Let $n'=n/p'$ and $S'$ be the image of $S$ under $f_{n|n'}$. Suppose at most one term of $S'$ is a unit OR suppose there is a prime divisor $p$ of $n/p'$ such that at least three terms of $S$ are coprime to $p$. Then $S$ is an $L(n;p')$-weighted zero-sum sequence. 
\end{lem}

\begin{obs}\label{obs3}
Let $A\su \Z_n$ and let $S$ be a sequence in $\Z_n$. Let $n=m_1m_2$ where $m_1$ and $m_2$ are coprime. For $i=1,2$, let $A_i\su\Z_{m_i}$ be given and let $S_i$ denote the image of the sequence $S$ under $f_{n|m_i}$. Let $\psi:U(n)\to U(m_1)\times U(m_2)$ be the isomorphism given by the Chinese remainder theorem. Suppose we have that $A_1\times A_2\su \psi(A)$. If $S_1$ is an $A_1$-weighted zero-sum sequence in $\Z_{m_1}$ and if $S_2$ is an $A_2$-weighted zero-sum sequence in $\Z_{m_2}$, then $S$ is an $A$-weighted zero-sum sequence in $\Z_n$. 
\end{obs}

The next two results follow from Theorems 5, 6, 7 and 8 of \cite{SKS3}. 

\begin{thm}\label{l}
Let $p$ be a prime divisor of $n$ where $n$ is a squarefree number which is not a product of two primes and every prime divisor of $n$ is at least 7. Then $D_{L(n;p)}(n)=\Omega(n)+1$ and $C_{L(n;p)}(n)=2^{\Omega(n)}$. 
\end{thm}

\begin{thm}\label{l2}
Let $p$ be a prime divisor of $n$ where $n$ is a product of two distinct primes which are at least 7. Then $D_{L(n;p)}(n)=4$ and $C_{L(n;p)}(n)=6$.
\end{thm}

\section{$D$-extremal sequences for $L(n;p')$}

\begin{rem}\label{ld}
Let $p$ be a prime divisor of $n$. As $L(n;p)\su U(n)$, an $L(n;p)$-weighted zero-sum subsequence is also a $U(n)$-weighted zero-sum subsequence. So if $n$ is such that $D_{U(n)}(n)=D_{L(n;p)}(n)$, then a $D$-extremal sequence for $U(n)$ is also a $D$-extremal sequence for $L(n;p)$. It is interesting to observe that the converse is also true for `most' squarefree numbers, as is shown in the next result. 
\end{rem}

\begin{thm}\label{dextl}
Let $n$ be a squarefree number such that every prime divisor of $n$ is at least 7. Suppose $p'$ is a prime divisor of $n$ and $\Omega(n)\neq 2,3$. Then $S$ is a $D$-extremal sequence for $L(n;p')$ if and only if $S$ is a $D$-extremal sequence for $U(n)$.  
\end{thm}

\begin{proof}
Let $n$ be a squarefree number such that  every prime divisor of $n$ is at least 7 and let $p'$ be a prime divisor of $n$. Suppose $\Omega(n)\neq 2$. By Theorems \ref{un} and \ref{l} we have that $D_{U(n)}(n)=D_{L(n;p)}(n)$, and so by Remark \ref{ld} it is enough to show that if $S$ is a $D$-extremal sequence for $L(n;p')$, then $S$ is a $D$-extremal sequence for $U(n)$.

Let $S=(x_1,\ldots,x_k)$ be a $D$-extremal sequence for $L(n;p')$. Then all terms of $S$ must be non-zero. If $\Omega(n)=1$, then $n=p'$. As $L(n;p')=U(n)$, it follows that $S$ is a $D$-extremal sequence for $U(n)$. Let $\Omega(n)\geq 4$. By Theorem \ref{l} we have $D_{L(n;p')}(n)=\Omega(n)+1$. Thus $S$ must have length $\Omega(n)$.

\begin{case}
There is a prime divisor $p$ of $n$ such that at most one term of $S$ is coprime to $p$.
\end{case}

Suppose all terms of $S$ are divisible by $p$. Let $n'=n/p$ and let $S'$ be the image of $S$ under $f_{n|n'}$. As $S'$ has length $\Omega(n)=\Omega(n')+1$, by Theorem \ref{sn} we see that $S'$ has an $S(n')$-weighted zero-sum subsequence. As $n$ is squarefree, so $p$ is coprime to $n'$. Thus, from Lemmas \ref{lifts'} and \ref{s2l} we get the contradiction that $S$ has an $L(n;p')$-weighted zero-sum subsequence.

Hence, there is exactly one term of $S$ which is not divisible by $p$. Let us assume that that term is $x_1$. Let $T:(x_2,\ldots,x_k)$ and let $T'$ be the image of $T$ under $f_{n|n'}$. Suppose $T'$ has an $S(n')$-weighted zero-sum subsequence. Then by Lemma \ref{s2l} we have $S(n')\su f_{n|n'}\big(L(n;p')\big)$ and so by Lemma \ref{lifts'} we get the contradiction that $S$ has an $L(n;p')$-weighted zero-sum subsequence. Thus, $T'$ is a sequence of length $\Omega(n')$ in $\Z_{n'}$ which does not have any $S(n')$-weighted zero-sum subsequence.

From Theorem \ref{sn} as $D_{S(n')}(n')=\Omega(n')+1$, it follows that $T'$ is a $D$-extremal sequence for $S(n')$. When $\Omega(n)\geq 4$, then $\Omega(n')\geq 3$ and so from  Theorem \ref{dexts}, we get that $T'$ is a $D$-extremal sequence for $U(n')$. So by Theorem 6 of \cite{AMP} we get that $S$ is a $D$-extremal sequence for $U(n)$.

\begin{case}
For every prime divisor $p$ of $n/p'$ exactly two terms of $S$ are coprime with $p$, and at least two terms of $S$ are coprime with $p'$.
\end{case}

Let $n'=n/p'$ and let $S'$ be the image of $S$ under $f_{n|n'}$. Suppose at most one term of $S'$ is a unit. By Lemma \ref{gl} we see that $S$ is an $L(n;p')$-weighted zero-sum sequence. Suppose at least two terms of $S'$ are units. By the assumption in this case we see that exactly two terms of $S'$ are units, say $x_{j_1}'$ and $x_{j_2}'$ and the other terms of $S'$ are zero. 

It follows that all terms of $S$ are divisible by $n'$ except $x_{j_1}$ and $x_{j_2}$. Let $j\neq j_1,j_2$. If $x_j$ is divisible by $p'$, we get the contradiction that $x_j=0$. So all the terms of $S^{(p')}$ are non-zero except possibly two terms. As the sequence $S$ has length at least 4, we can find a subsequence $T$ of $S$ having length at least two which does not contain the terms $x_{j_1}$ and $x_{j_2}$. 

As all the terms of $T^{(p')}$ are non-zero and as $T^{(p')}$ has length at least 2, by Lemma 2.1 (ii) of \cite{sg} we see that $T^{(p')}$ is a $U(p')$-weighted zero-sum sequence. Also all the terms of $T$ are divisible by $n'$.  Hence, by taking $d=n'$ in Lemma \ref{lifts'} and by Lemma \ref{u2l} we see that $T$ is an $L(n;p')$-weighted zero-sum subsequence of $S$. So we get the contradiction that $S$ has an $L(n;p')$-weighted zero-sum subsequence.  

\begin{case}
For every prime divisor $p$ of $n$ at least two terms of $S$ are coprime to $p$, and there is a prime divisor $p$ of $n/p'$, such that at least three terms of $S$ are coprime to $p$.
\end{case}

In this case, by Lemma \ref{gl} we get the contradiction that $S$ is an $L(n;p')$-weighted zero-sum sequence.  
\end{proof}

\begin{lem}\label{s2l3}
Let $n=p'p\,q$ where $p',p\,,q$ are distinct primes and $n'=n/p$.  Then $U(n')\su f_{n|n'}\big(L(n;p')\big)$.
\end{lem}

\begin{proof}
Let $n'=n/p$. As $p$ is coprime with $n'$, by the Chinese remainder theorem we have an isomorphism $\psi:U(n)\to U(n')\times U(p)$. Let $b\in U(n')$ and let $c\in U(p)$ such that $\Big(\dfrac{c}{p}\Big)=\Big(\dfrac{b}{q}\Big)$. Let $a\in U(n)$ such that $\psi(a)=(b,c)$. Then $a\in L(n;p')$ as
$$\Big(\dfrac{a}{n}\Big)=\Big(\dfrac{b}{n'}\Big)\Big(\dfrac{c}{p}\Big)=\Big(\dfrac{b}{p'q}\Big)\Big(\dfrac{b}{q}\Big)=\Big(\dfrac{b}{p'}\Big)=\Big(\dfrac{a}{p'}\Big).$$ As $f_{n|n'}(a)=b$, we get that $b\in f_{n|n'}\big(L(n;p')\big)$. 
\end{proof}

\begin{thm}\label{extl3}
Let $n$ be squarefree such that every prime divisor of $n$ is at least 7. Suppose $p'$ is a prime divisor of $n$ and $\Omega(n)=3$. Then a sequence in $\Z_n$ is a $D$-extremal sequence for $L(n;p')$ if and only either it is a $D$-extremal sequence for $U(n)$ or it is a permutation of a sequence $(x_1,x_2,x_3)$ which has one of the following forms.

\begin{itemize}
\item
The image of the sequence $(x_1,x_2)$ under $f_{n|n'}$ is a $D$-extremal sequence for $S(n')$ and $x_3$ is a non-zero multiple of $n'$, where $n'=n/p'$.

\item
The only term of $S$ which is not divisible by $p'$ is $x_1$ and the image of the sequence $(x_2,x_3)$ under $f_{n|n'}$ is a $D$-extremal sequence for $S(n')$ which is not a $D$-extremal sequence for $U(n')$, where $n'=n/p'$.
\end{itemize}
\end{thm}

\begin{proof}
Let $n$ and $p'$ be as in the statement of the theorem. Suppose the sequence $S=(x_1,x_2,x_3)$ is a $D$-extremal sequence for $L(n;p')$. Then all terms of $S$ must be non-zero.

\begin{case}
There is a prime divisor $p$ of $n$ such that at most one term of $S$ is coprime to $p$.
\end{case}

Suppose all terms of $S$ are divisible by $p$. We use a similar argument as in this subcase of the proof of Theorem \ref{dextl} to get the contradiction that $S$ has an $L(n;p')$-weighted zero-sum subsequence.

Suppose exactly one term of $S$ is not divisible by $p$. Let us assume that that term is $x_1$ and let $n'=n/p$. Let $T'$ be the image of $T=(x_2,x_3)$ under $f_{n|n'}$. By a similar argument as in this subcase of the proof of Theorem \ref{dextl}, we see that $T'$ is a $D$-extremal sequence for $S(n')$. 

Suppose $p\neq p'$. We claim that $T'$ is infact a $D$-extremal sequence for $U(n')$. As $\Omega(n')=2$, by Theorem \ref{un} we have $D_{U(n')}(n')=3$. So it is enough to show that $T'$ does not have any $U(n')$-weighted zero-sum subsequence. As $n$ is squarefree and $\Omega(n)=3$, by Lemma \ref{s2l3} we have $U(n')\su f_{n|n'}\big(L(n;p')\big)$. 

So if $T'$ has a $U(n')$-weighted zero-sum subsequence, by Lemma \ref{lifts'} we get the contradiction that $T$ (and hence $S$) has an $L(n;p')$-weighted zero-sum subsequence. Hence our claim is true. Thus, by Theorem 6 of \cite{AMP} we see that $S$ is a $D$-extremal sequence for $U(n)$. 

\begin{case}
For every prime divisor $p$ of $n/p'$ exactly two terms of $S$ are coprime to $p$, and at least two terms of $S$ are coprime to $p'$.
\end{case}

Let $n'=n/p'$. Let $S'$ be the image of the sequence $S$ under  $f_{n|n'}$. Suppose at most one term of $S'$ is a unit. By Lemma \ref{gl}, we see that $S$ is an $L(n;p')$-weighted zero-sum sequence. So we can assume that at least two terms of $S'$ are units, say $x_1'$ and $x_2'$. By the assumption in this case, we see that $x_1'$ and $x_2'$ are units and $x_3'$ is zero. So we see that $x_3$ is divisible by $n'$. As $x_3\neq 0$, it follows that $x_3$ is coprime to $p'$. 

If $(x_1',x_2')$ has an $S(n')$-weighted zero-sum subsequence, then the sequence $S'$ is an $S(n')$-weighted zero-sum sequence as $x_3'=0$. Also the sequence $S^{(p')}$ is a $U(p')$-weighted zero-sum sequence. Let $\psi:U(n)\to U(n')\times U(p')$ be the isomorphism given by the Chinese remainder theorem. From Lemma \ref{gl'}, we have $S(n')\times U(p')\su \psi(\big(L(n;p')\big)$. So by Observation \ref{obs3} we get the contradiction that $S$ is an $L(n;p')$-weighted zero-sum sequence.

Hence, $(x_1',x_2')$ does not have any $S(n')$-weighted zero-sum subsequence. By Theorem \ref{sn} we have $D_{S(n')}(n')=3$. So the sequence $(x_1',x_2')$ is a $D$-extremal sequence for $S(n')$.  

\begin{case}
For every prime divisor $p$ of $n$ at least two terms of $S$ are coprime to $p$, and there is a prime divisor $p$ of $n/p'$ such that at least three terms of $S$ are coprime to $p$.
\end{case}

In this case by Lemma \ref{gl}, we get the contradiction that $S$ is an $L(n;p')$-weighted zero-sum sequence. 
\end{proof}

\begin{rem}
We omit the proof of the other implication in the statement of this theorem as well as the proofs of Theorems \ref{extl2} and Theorem \ref{lext2'} to avoid making the paper lengthy.  
\end{rem}

\begin{thm}\label{extl2}
Let $n=p'q$ where $p',q$ are distinct primes which are at least 7. Then a sequence $S$ in $\Z_n$ is a $D$-extremal sequence for $L(n;p')$ if and only if $S$ is a permutation of a sequence $(x_1,x_2,x_3)$ which has one of the following forms.

\begin{itemize}
\item 
The image of the sequence $(x_1,x_2)$ under $f_{n|q}$ is a $D$-extremal sequence for $Q_q$ and $x_3$ is a non-zero multiple of $q$. 

\item
The only term of $S$ which is coprime to $p'$ is $x_1$ and the image of the sequence $(x_2,x_3)$ under $f_{n|q}$ is a $D$-extremal sequence for $Q_q$. 
\end{itemize}
\end{thm}

\begin{rem}
By Theorems \ref{un} and \ref{l}, we see that when $n=p'q$ where $p',q$ are distinct primes which are at least 7, we have $D_{U(n)}(n)=3$ and $D_{L(n;p')}=4$. So we cannot compare the $D$-extremal sequences for $U(n)$ with the $D$-extremal sequences for $L(n;p')$. We also observe that the sequences which have been mentioned in Theorem \ref{extl2} are a proper subcollection of those which have been mentioned in Theorem \ref{extl3}.
\end{rem}

\section{$C$-extremal sequences for $L(n;p')$}

\begin{rem}\label{lc}
Let $p$ be a prime divisor of $n$. As $L(n;p)\su U(n)$, an $L(n;p)$-weighted zero-sum subsequence is also a $U(n)$-weighted zero-sum subsequence. So if $n$ is such that $C_{U(n)}(n) = C_{L(n;p)}(n)$, then a $C$-extremal sequence for $U(n)$ is also a $C$-extremal sequence for $L(n;p)$. It is interesting to observe that the converse is also true for `most' squarefree numbers as is shown in the next result. 
\end{rem}

\begin{thm}\label{cextl}
Let $n$ be a squarefree number such that every prime divisor of $n$ is at least 7 and let $p'$ be a prime divisor of $n$. Suppose $\Omega(n)\neq 2$. Then the $C$-extremal sequences for $L(n;p')$ are the same as the $C$-extremal sequences for $U(n)$.  
\end{thm}

\begin{proof}
Let $n$ be a squarefree number such that  every prime divisor of $n$ is at least 7 and let $p'$ be a prime divisor of $n$. Suppose $\Omega(n)\neq 2$. By Theorems \ref{un} and \ref{l} and Remark \ref{lc}, it is enough to show that if $S$ is a $C$-extremal sequence for $L(n;p')$, then $S$ is a $C$-extremal sequence for $U(n)$.

Let $S=(x_1,\ldots,x_k)$ be a $C$-extremal sequence for $L(n;p')$. Then all terms of $S$ must be non-zero. When $n$ is a prime, then $n=p'$ and $L(n;p')=U(n)$. So $S$ is a $C$-extremal sequence for $U(n)$. Let $\Omega(n)\geq 3$. By Theorem \ref{l} we have $C_{L(n;p')}(n)=2^{\Omega(n)}$. So $S$ has length $l=2^{\Omega(n)}-1$.  

\begin{case}
There is a prime divisor $p$ of $n$ such that at most one term of $S$ is coprime to $p$. 
\end{case}

Suppose $x_{k+1}$ is divisible by $p$ where $k+1=(l+1)/2$. Then we can find a subsequence $T$ having consecutive terms of $S$ of length $k+1=2^{\Omega(n')}$ such that all the terms of $T$ are divisible by $p$. Let $n'=n/p$ and let $T'$ be the image of $T$ under $f_{n|n'}$. 

As $\Omega(n')=\Omega(n)-1\geq 2$ and as $T'$ has length $2^{\Omega(n')}$, by Theorem \ref{sn} we see that $T'$ has an $S(n')$-weighted zero-sum subsequence of consecutive terms. By Lemma \ref{s2l} we get that $S(n')\su f_{n|n'}\big(L(n;p')\big)$. So by Lemma \ref{lifts'} we get the contradiction that $T$ (and hence $S$) has an $L(n;p')$-weighted zero-sum subsequence of consecutive terms. 

So $x_{k+1}$ is not divisible by $p$. Let $S_1=(x_1,\ldots,x_k)$ and $S_2=(x_{k+2},\ldots,x_l)$. Let $S_1'$ and $S_2'$ be the images of $S_1$ and $S_2$ respectively under $f_{n|n'}$. Suppose $S_1'$ has an $S(n')$-weighted zero-sum subsequence of consecutive terms. By Lemma \ref{s2l} we get that $S(n')\su f_{n|n'}\big(L(n;p')\big)$. So by Lemma \ref{lifts'} we get the contradiction that $S_1$ (and hence $S$) has an $L(n;p')$-weighted zero-sum subsequence of consecutive terms. 

So $S_1'$ does not have any $S(n')$-weighted zero-sum subsequence of consecutive terms. From Theorem \ref{sn} as $\Omega(n')\geq 2$ we have that $C_{S(n')}(n')=2^{\Omega(n')}$. As $S_1'$ has length $k=2^{\Omega(n')}-1$, it follows that $S_1'$ is a $C$-extremal sequence for $S(n')$. As $\Omega(n')\geq 2$, from Theorem \ref{cexts} we see that $S_1'$ is a $C$-extremal sequence for $U(n')$. Similarly $S_2'$ is also a $C$-extremal sequence for $U(n')$. By Theorem 5 of \cite{SKS2} we get that $S$ is a $C$-extremal sequence for $U(n)$.

\begin{case}
For every prime divisor $p$ of $n/p'$ exactly two terms of $S$ are coprime with $p$, and at least two terms of $S$ are coprime with $p'$.
\end{case}

We use a similar argument as in this case of the proof of Theorem \ref{dextl}. We just observe that as the sequence $S$ has length at least 7, we can find a subsequence $T$ having consecutive terms of $S$ and having length at least two, which does not contain the terms $x_{j_1}$ and $x_{j_2}$. 

\begin{case}
For every prime divisor $p$ of $n$ at least two terms of $S$ are coprime to $p$, and there is a prime divisor $p$ of $n/p'$, such that at least three terms of $S$ are coprime to $p$.
\end{case}

In this case, by Lemma \ref{gl} we get the contradiction that $S$ is an $L(n;p')$-weighted zero-sum sequence. 
\end{proof}

\begin{thm}\label{lext2'}
Let $n=p'q$ where $p'$ and $q$ are distinct primes which are at least 7. Suppose $S=(x_1,x_2,x_3,x_4,x_5)$ is a sequence in $\Z_n$. Then $S$ is a $C$-extremal sequence for $L(n;p')$ if and only if $S$ has either of the following two forms. 

\begin{itemize}
\item 
The terms $x_1,x_3$ and $x_5$ are non-zero multiples of $q$ and the image of the sequence $(x_2,x_4)$ under $f_{n|q}$ is a $C$-extremal sequence for $Q_q$. 

\item
The only term of $S$ which is coprime to $p'$ is $x_3$ and the images of the sequences $(x_1,x_2)$ and $(x_4,x_5)$ under $f_{n|q}$ are $C$-extremal sequences for $Q_q$.
\end{itemize}
\end{thm}

\section{Concluding remarks}

Let $n$ and $p'$ be as in the statement of Theorem \ref{lext2'}. From Theorems \ref{un} and \ref{l}, we have $D_{U(n)}(n)=3$ and $D_{L(n;p')}(n)=4$, and so we cannot compare the $D$-extremal sequences for $U(n)$ with the $D$-extremal sequences for $L(n;p')$. Also as $C_{U(n)}(n)=4$ and $C_{L(n;p')}(n)=6$, we cannot compare the corresponding $C$-extremal sequences. 

When $\Omega(n)=2$, from Theorem \ref{dexts'} we see that there exist $D$-extremal sequences for $S(n)$ which are not $D$-extremal sequences for $U(n)$ and when $\Omega(n)=3$, from Theorem \ref{extl3} we see that there exist $D$-extremal sequences for $L(n;p')$ which are not $D$-extremal sequences for $U(n)$. The following questions can be investigated. 

\begin{itemize}
\item 
Can we determine the value of $D_{S(n)}(n)$ and characterize the $D$-extremal sequences for $S(n)$ for general $n$.

\item 
If $n$ is such that $D_{U(n)}(n)=D_{S(n)}(n)$, can we say that a sequence in $\Z_n$ which does not have any $S(n)$-weighted zero-sum subsequence will not have any $U(n)$-weighted zero-sum subsequence. 

\item
We can ask analogues of the above questions for the constants $C_{S(n)}(n)$ and $C_{U(n)}(n)$.
\end{itemize}

{\bf Acknowledgement.}
Santanu Mondal would like to acknowledge CSIR, Govt. of India, for a research fellowship. We thank the referee for taking the time to go through this paper and for the corrections and suggestions.

\end{document}